\documentclass[11pt,reqno]{amsart}
\usepackage{latexsym}
\RequirePackage{amsmath}   \RequirePackage{amssymb}
\RequirePackage{amsthm}   \RequirePackage{epsfig}
\theoremstyle{plain}
\newtheorem{theorem}{Theorem}
\newtheorem{corollary}{Corollary}
\newtheorem{lemma}{Lemma} 
\newtheorem*{corollary*}{Corollary}
\newtheorem{theoremO}{Theorem}

\newtheorem*{conjecture*}{Conjecture}
\newtheorem*{proposition*}{Proposition}
\newtheorem*{lemma*}{Lemma}

\theoremstyle{definition}

\theoremstyle{remark}

\newcommand{\SC}{{\mathbb C}}  \newcommand{\SD}{{\mathbb D}}  
\newcommand {\SR}{{\mathbb R}}  \newcommand{\ST}{{\mathbb T}}  
\newcommand{\al}{\alpha}  \newcommand{\ga}{\gamma} \newcommand{\de}{\delta} 
  \newcommand{\ve}{\varepsilon}  \newcommand{\ze}{\zeta}
\def\t{\theta}  \newcommand{\vt}{\vartheta}  \newcommand{\la}{\lambda}
  \newcommand{\vp}{\varphi}  
  \newcommand{\Om}{\Omega}

  \newcommand{\cP}{{\mathcal P}}    

\newcommand{\be}{\begin{equation}}
\newcommand{\ee}{\end{equation}}
\newcommand{\bea}{\begin{eqnarray}}
\newcommand{\eea}{\end{eqnarray}}

\begin{document}

\title{On Hardy spaces, univalent functions and the second coefficient}

\author[M. Chuaqui]{Martin Chuaqui}
\address{Facultad de Matem\'aticas, Pontificia Universidad Cat\'olica de Chile, Casilla 306, Santiago 22, Chile.}  \email{mchuaqui@uc.cl}

\author[I. Efraimidis]{Iason Efraimidis}
\address{Departamento de Matem\'aticas, Universidad Aut\'onoma de Madrid, 28049 Madrid, Spain} \email{iason.efraimidis@uam.es}

\author[R. Hernández]{Rodrigo Hernández}
\address{Facultad de Ciencias y Tecnolog\'ia, Universidad Adolfo Ib\'a\~nez, Vi\~na del Mar, Chile.} \email{rodrigo.hernandez@uai.cl}

\keywords{Hardy class, univalent functions, convex, starlike, close-to-convex, coefficients \\ Martin Chuaqui: https://orcid.org/0000-0002-7118-2812 \\ Iason Efraimidis: https://orcid.org/0000-0002-0252-5607 \\ Rodrigo Hernández: https://orcid.org/0000-0002-2787-4598}

\maketitle

\begin{abstract}
We consider normalized univalent functions with prescribed second Taylor coefficient $a_2$. For convex functions $f$ we study the Hardy spaces to which $f$ and $f'$ belong, refining in particular on a theorem of Eenigenburg and Keogh, and give a sharp asymptotic estimate and an explicit uniform bound for their coefficients. Relating the lower order of a convex function to the angle at infinity of its range we deduce that its range lies always in some sector of aperture $|a_2|\pi$. We give sharp smoothness conditions on the boundary for convex functions with prescribed second coefficient. 

We find the sharp Hardy space estimates for $f$ and $f'$ when $f$ belongs to other geometric subclasses, such as those of starlike, close-to-convex, convex in one direction, convex in the positive direction and typically real funtions. We extend a theorem of Lohwater, Piranian and Rudin, in which a univalent function whose derivative has radial limits almost nowhere is constructed, by showing that this pathological behavior can be obtained for any prescribed value of the second coefficient, in particular, manifesting itself arbitrarily close to the Koebe function. 
\end{abstract}

\section{Introduction} 

Let $\SD$ denote the unit disk and $S$ the class of univalent functions 
$$
f(z) = z + \sum_{n=2}^{\infty} a_n z^n, \qquad  z\in\SD. 
$$
In view of the implications of Bieberbach's inequality $|a_2|\leq 2$ in covering, growth and distortion theorems, among others, it is of interest to study constraint extremal problems having $|a_2|$ as a prescribed parameter. Results on problems of this kind appeared as early as 1920 with Gronwall's paper \cite{Gr20}, while a list of references, inevitably incomplete, is \cite{Ah70, ASS99, Bi76, Eh74, Fi67, Je54, Lee73, Lew74, Ne70, SSW80, Te70}; see also the recent survey \cite{RS22}. 

The integral means of a function $f:\SD\to\SC$ are typically denoted by 
$$
M_p(r,f) =  \left\{ \frac{1}{2\pi}\int_0^{2\pi} |f(re^{i\t})|^p d\t \right\}^{1/p}, \qquad p>0,
$$
and its maximum modulus by
$$
M_{\infty}(r,f) =  \max_{|z|=r} |f(z)|. 
$$
The Hardy space $H^p$ is defined as the space of holomorphic functions $f$ in $\SD$ satisfying 
$$
\sup_{r<1} M_p(r,f) <\infty,
$$ 
for $0<p\leq \infty$. The study of the integral means of univalent functions and their derivatives is a classical topic. For example, in 1927 Prawitz showed that $S\subset H^p$ for all $p<1/2$. One can find a wealth of information on the growth of the means of functions in $S$ in Pommerenke's book \cite[\S 5.1]{Po75}, while Baernstein's celebrated inequality 
$$
M_p(r,f) \leq  M_p(r,k), \qquad r\in(0,1),
$$ 
for all $p \in (0,\infty)$ and $f \in S$, where $k(z) = \frac{z}{(1-z)^2}$ is the Koebe function, can be found in Chapter 7 of Duren's book \cite{Du2}. Further references on this subject, including characterizations of univalent functions in Hardy spaces in terms of geometric properties of their range, as well as applications to Operator Theory, can be found in \cite{BGP04, CCKR24, EK70, Han70, Ka20, Ka23, KK21, KS09, KS11, Leu79, PC96, PC97a, PC97b}. 

In this article we go a step further by analyzing the integral means of univalent functions with \emph{prescribed} $|a_2|$, a question not addressed in the literature. We provide a comprehensive study of Hardy space estimates for $f$ and $f'$ when $f$ belongs to the standard geometric subclasses, paying special attention to convex functions by studying their coefficients, their smoothness at the boundary and the angle at infinity of their range. On the other hand, we show that certain pathological behavior of the derivative of a univalent function, first shown to exist by Lohwater, Piranian and Rudin, can be realized arbitrarily close to the Koebe function.

\vskip.2cm 
\noindent \textbf{Convex functions, Hardy spaces and coefficients.} Our starting point is a theorem of Eenigenburg and Keogh~\cite{EK70}, stating that if $f$ is in $C$, the subclass of $S$ consisting of functions whose range is a \emph{convex} set, and if $f$ is not a rotation of the half-plane function $z\mapsto \frac{z}{1-z}$ then there exists $\de>0$ for which $f' \in H^{ \frac{1}{2}+ \de}$. Our first result makes this precise in terms of $|a_2|$.

\begin{theorem} \label{thm-convex-deriv-Hardy}
If $f\in C$ then $f'\in H^p$ for all $p<\frac{1}{1+|a_2|}$ and $f\in H^q$ for all $q<\frac{1}{|a_2|}$. Both estimates are sharp. 
\end{theorem}

The sharpness comes from the sector function, with $\al\in[0,1]$, given by 
\be \label{formula-sector}
s_{\al}(z) = \frac{1}{2\al} \left[ \left( \frac{1+z}{1-z} \right)^{\al} - 1\right]  = z +\al z^2 + \frac{1+2\al^2}{3} z^3 + \ldots 
\ee
(understood as $z\mapsto \frac{1}{ 2 }\log\frac{ 1+ z }{ 1- z }$ when $\al=0$), which maps $\SD$ onto a sector of opening $\al\pi$. 

For the coefficients of functions in $C$ we find their sharp order of growth and an explicit bound. Recall that $|a_2|\leq1$ with equality only for rotations of the half-plane function $s_1(z)=\frac{z}{1-z}$. 

\begin{theorem} \label{thm-conv-n-coeff}
If $f\in C$ then $a_n\in O(n^{|a_2|-1})$, which is sharp, and 
$$
|a_n| \leq \exp\left( \frac{|a_2|^2-1}{2} \right), \qquad n\geq2.
$$
\end{theorem}

Following Cruz and Pommerenke \cite{CP07} we define the \emph{lower order} of a function $f$ in $\SD$ as 
$$
\beta = \inf_{\ze\in\SD} |A_f(\ze)|, 
$$
where the expression 
$$
A_f(\ze) = \frac{1}{2} (1-|\ze|^2) \frac{f''(\ze)}{f'(\ze)} - \overline{\ze},  \qquad \ze\in\SD, 
$$
is the second coefficient of the Koebe transform of $f$, given by 
\be \label{form-Koebe-transf}
F_{\ze}(z) = \frac{ f\left( \frac{\ze+z}{1+\overline{\ze}z} \right) - f(\ze) }{ (1-|\ze|^2) f'(\ze) } = z + A_f(\ze) z^2 + \ldots, \qquad \ze, z \in\SD. 
\ee
Clearly, $ 0\leq \beta\leq |a_2|$, while we will prove in Section~\ref{sect-convex-Hardy-coeff} that 
\vskip.22cm
\begin{quote} 
\textit{$\beta=|a_2|$ if and only if either $a_2=0$ or $f$ is a rotation of $s_\al$.}
\end{quote}
\vskip.17cm
We have the following improvement of Theorem~\ref{thm-convex-deriv-Hardy}.

\vskip.2cm
\begin{theorem} \label{thm-convex-Hardy-low-ord}
If $f\in C$ then $f'\in H^p$ for all $p<\frac{1}{1+\beta}$ and $f\in H^q$ for all $q<\frac{1}{\beta}$.  
\end{theorem}

\vskip.2cm 
\noindent \textbf{Convex functions and the angle at infinity.} We define the \emph{angle at infinity} of a convex domain $\Om$ as 
\be \label{def-angle-infinity}
\Theta(\Om)  = \inf \{ \, \t \in[0,\pi] \; : \; \Om \; \text{is contained in a sector of aperture} \; \t \; \}; 
\ee
see \cite[\S 3]{CHLM23}. An equivalent definition in terms of half-tangents is given in Section~\ref{sect-convex-angle}. We regard a sector of aperture zero as an infinite strip. Clearly, $\Theta=0$ for bounded domains, infinite half-strips and strips. Note that the infimum in \eqref{def-angle-infinity} is not always attained. For example, $\Theta=0$ for a parabolic region but no parallel strip contains it. Simple modifications provide similar examples for other values of $\Theta$ in $(0,\pi)$. On the other hand, if $\Theta=\pi$ then this value is always attained by any of the supporting half-planes of $\Om$. In fact, we will prove that in this case $\Om$ itself must be a half-plane.

Our main theorem relates the angle at infinity with the lower order as follows. 

\begin{theorem} \label{thm-angle-l-order}
If $f\in C$ then $\Theta\big(f(\SD)\big) = \beta \pi$. 
\end{theorem}

From this, we deduce some corollaries. We will be needing the following theorem of Chuaqui and Osgood \cite{CO95} and Fournier, Ma and Ruscheweyh \cite{FMR98}. We provide a new proof of this in Section~\ref{sect-CO-FMR-measures}. 

\begin{theoremO}[\cite{CO95, FMR98}] \label{thm-a=0-CO-FMR}
If $f\in C$ has $a_2=0$ then $f$ is either bounded or a rotation of $s_0(z)= \frac{1}{2} \log\frac{1+z}{1-z}$. 
\end{theoremO}

We now deduce our first corollary from Theorem~\ref{thm-angle-l-order}.

\begin{corollary}\label{cor-sect-a2}
If $f\in C$ then its image $f(\SD)$ lies in some sector of aperture $|a_2|\pi$. 
\end{corollary}

Indeed, if $a_2=0$ then the corollary holds for the two cases given in Theorem~\ref{thm-a=0-CO-FMR}. If $a_2\neq0$ then, by the aforementioned characterization of the case when $\beta$ equals $|a_2|$, we see that $f$ is either a rotation of a sector function $s_\al$ or $\Theta < |a_2| \pi$, again, two cases for which the corollary holds. 

Corollary~\ref{cor-sect-a2} is sharp in view of the functions $s_\al$. Moreover, the aperture $|a_2|\pi$ cannot be improved to $\beta \pi$ since, as discussed earlier, the infimum in \eqref{def-angle-infinity} is not always attained.

By subordination, if $f$ is holomorphic in $\SD$ and its range lies in some sector of aperture $\t$ then $f \in H^q$, for $q<\frac{\pi}{\t}$; see Exercise 2 in \cite[Chap.~1]{Du1}. Hence, from Corollary~\ref{cor-sect-a2} we deduce the following, which is the Hardy space estimate for the function $f$ in Theorem~\ref{thm-convex-deriv-Hardy}, thus given an alternative geometric proof. 

\begin{corollary} 
If $f\in C$ then $f\in H^q$ for all $q<\frac{1}{|a_2|}$. 
\end{corollary}

Finally, since for functions in $C$ we have $|a_2|=1$ only for half-plane mappings, we obtain the following. 

\begin{corollary}
If $f\in C$ with $\Theta\big(f(\SD)\big) =\pi$ then $f$ is a half-plane mapping. 
\end{corollary}

\vskip.2cm 
\noindent \textbf{Convex functions and smoothness at the boundary.} Let $\Lambda_t$ denote the class of functions $\phi$ defined on $[0,2\pi]$ which satisfy the Lipschitz condition of order $t\in(0,1]$:  
$$
|\phi(x) - \phi(y) | \leq c \, |x-y|^t, 
$$
for some $c>0$. Clearly, $s<t$ implies $\Lambda_s \supset \Lambda_t$. Moreover, we consider the wider class $\Lambda_t^p$, for $t \in (0, 1]$ and $p\geq 1$, of functions $\phi \in L^p(0,2\pi)$ that satisfy the integral Lipschitz condition 
$$
 \left( \int_0^{2\pi} |\phi(\t+h) - \phi(\t) |^p d\t \right)^{1/p} \leq c \, h^t, \quad \; \text{for} \;\; h>0, 
$$
where $c>0$ is some constant. It is easy to see that if $s<t$ then $\Lambda_s^p \supset \Lambda_t^p$, while also $p<q$ implies $\Lambda_t^p \supset \Lambda_t^q$. Here is our main theorem. 

\begin{theorem} \label{thm-C-smooth-general}
If $f\in C$  is not a rotation of the half-plane function $s_1(z) = \frac{z}{1-z}$ then $f(e^{i\t}) \in \Lambda_t^p$ with $t= \frac{1}{p}-|a_2|$, for every $1\leq p < \frac{1}{|a_2|}$. This is sharp in the sense that for every $\al\in[0,1)$ there exists $f\in C$ with $\al = |a_2|$, for which $f(e^{i\t}) \notin \Lambda_t^p$ with $t= \frac{1}{p}-\al+\ve$, for every $1\leq p < \frac{1}{\al}$ and $\ve>0$. 

\end{theorem}

In the case $a_2=0$ we are able to prove a stronger Lipschitz condition on the boundary.

\begin{theorem} \label{thm-C-smooth-a2=0}
Let $f\in C$ with $a_2=0$. If $f$ is not a rotation of the parallel strip function $s_0(z)=\frac{1}{2}\log\frac{1+z}{1-z}$ then $f(e^{i\t}) \in \Lambda_t$ for $t=\frac{1-3|a_3|}{3(1+|a_3|)}$. 
\end{theorem}

\vskip.2cm 
\noindent \textbf{Other geometric subclasses of $S$ and Hardy spaces.} Theorem~\ref{thm-convex-deriv-Hardy} takes the following form in $S^*$, the subclass of $S$ consisting of functions whose range is \emph{starlike} with respect to the origin. This refines another theorem of Eenigenburg and Keogh~\cite{EK70}.

\begin{theorem} \label{thm-Hardy-star}
If $f\in S^*$ then $f'\in H^p$ for all $p<\frac{2}{4+|a_2|}$ and $f\in H^q$ for all $q<\frac{2}{2+|a_2|}$. Both estimates are sharp. 
\end{theorem}

For the subclass $K$ of \emph{close-to-convex} functions, a theorem of Leung \cite{Leu79} provides the sharp estimate: if $f \in K$ then the derivative $f'$ belongs to $H^p$ for $p<1/3$, which is sharp. The following theorem provides a slight improvement of this when $a_2=0$ while showing that it remains sharp when $a_2\neq 0$.

\begin{theorem} \label{thm-close-to-convex-Hardy}
Let $f\in K$. We have that
\begin{itemize}
\item[(i)] if $a_2=0$ then $f'\in H^{1/3}$ and $f\in H^{1/2}$, while 
\item[(ii)] if $a_2\neq0$ then $f'\in H^p$ for all $p<1/3$ and $f\in H^q$ for all $q<1/2$. 
\end{itemize}
All four estimates are sharp.

\end{theorem}

Let $R$ denote the subclass of $S$ consisting of functions $f$ whose range is \emph{convex in one direction}, that is, there exists a direction $\t\in[0,\pi)$ such that for each line $L$ parallel to the line $\{te^{i\t} : t\in\SR \}$ the intersection $L\cap f(\SD)$ is either empty or a connected set. Domains that are convex in one direction have been of importance in the theory of harmonic mappings; see \cite[Chap.~3]{Du3}.

Let $R^+ $ denote the class of functions $f$ in $S$ whose range is \emph{convex in the positive direction}, which means that for each $w\in f(\SD)$ the half-line $\{w+t : t\geq 0 \}$ lies in $f(\SD)$. This is an important class in the theory of semigroups of holomorphic functions in the disk; see \cite{ES10}. 

It is well known that $R^+ \subset  R \subset K$, so that the estimates in Theorem~\ref{thm-close-to-convex-Hardy} apply to functions in $R^+$ and $R$. Moreover, these estimates are optimal when $a_2\neq 0$. However, for $a_2=0$ the optimal exponent exhibits a jump.

\begin{theorem} \label{Hardy-convex-one-direction}
Let $f\in S$ and assume that $a_2=0$. We have that
\begin{itemize}
\item[(i)] if $f\in R$ then $f'\in H^p$ for all $p<1/2$ and $f\in H^q$ for all $q<1$, while 
\item[(ii)] if $f\in R^+$ then $f'\in H^p$ for all $p<1$ and $f\in H^q$ for all $q<\infty$.  
\end{itemize}
All four estimates are sharp.
\end{theorem}

We now show that the improvement of the Hardy space estimate for close-to-convex functions given in Theorem~\ref{thm-close-to-convex-Hardy} for the case $a_2=0$ does not persist in the class $S$, in fact, not even in the class $S_\SR$, the subclass of $S$ consisting of functions that have \emph{real coefficients}. On the other hand, it is not difficult to see that Prawitz' Hardy space estimate for functions in $S$ also holds in the class of \emph{typically real} functions $T$, which also contains non-univalent functions.

\begin{theorem}\label{thm-f-Hardy-typ-real}
It holds that $T \subset H^p$ for all $p<1/2$. For every $\al \in [0,2]$ there exists $f\in S_\SR$ with $a_2=\al$ and such that $f\notin H^{1/2}$. 
\end{theorem}

\vskip.2cm 
\noindent \textbf{Integrability of the derivative of univalent functions.} Recall that the \emph{Nevanlinna class} $N$ consists of holomorphic functions $f$ in $\SD$ having bounded characteristic:
$$
\sup_{r\in(0,1)} \, \int_0^{2\pi} \log^+|f(re^{i\t})| d\t < \infty,
$$
where $\log^+x = \max\{0,\log x\}$. It is known that $N$ is wider than all Hardy spaces, that is, $\cup_{p>0} H^p \subset N$; see \cite[\S 2.1]{Du1}. 

The Bloch-Nevanlinna conjecture, asserting that $f' \in N$ whenever $f \in N$, has been proven false by various authors. A striking counterexample was provided by Lohwater, Piranian and Rudin \cite{LPR55}. They showed that for a suitably chosen increasing sequence $\{n_p\}$ of positive integers, the function 
$$
\Phi(z) = \int_0^z \exp{ \bigg( \tfrac{1}{2}\sum_{p=1}^\infty w^{n_p} \bigg) } dw 
$$
is holomorphic in $\SD$, continuous and injective in $\overline{\SD}$, and for almost all $\t$ it satisfies 
$$
0=\liminf_{r\to 1} |\Phi'(re^{i\t})| \,< \, \limsup_{r\to 1} |\Phi'(re^{i\t})| = + \infty
$$
and 
$$
-\infty =\liminf_{r\to 1} \, {\rm arg} \, \Phi'(re^{i\t})  \,< \,  \limsup_{r\to 1} \, {\rm arg}\, \Phi'(re^{i\t})  = + \infty. 
$$
This means that $\Phi\in H^\infty\cap S_\SR$ but $\Phi'\notin N$. 

It's easy to see that $a_2(\Phi) = \frac{1}{4}$ (it relies on the choice $n_1=1$ made in \cite{LPR55}). We mention that the construction of $\Phi$ can be modified so that the second coefficient takes any value in the interval $[0,\frac{\pi}{8})$, but we will not supply a proof for this. Instead, composing the Koebe function with an integral transform of the function $\Phi$, we will prove the following.

\begin{theorem} \label{thm-Lohwater-Pir-Rudin}
For every $\al \in[0,2)$ there exists $f \in S_\SR$ such that $a_2 =\al$ and $f'\notin N$. 
\end{theorem}

It is interesting to note that in view of this theorem, the pathological behavior of the function $\Phi$ can be encountered arbitrary close to the Koebe function $k$, whose derivative $k'(z)=\frac{1+z}{(1-z)^3}$ belongs to $H^p$ for all $p<1/3$, therefore to $N$ as well.

\vskip.1cm
The organization of the article is as follows. We prove Theorems~\ref{thm-convex-deriv-Hardy}, \ref{thm-conv-n-coeff} and \ref{thm-convex-Hardy-low-ord} in Section~\ref{sect-convex-Hardy-coeff}, Theorem~\ref{thm-angle-l-order} in Section~\ref{sect-convex-angle}, Theorems~\ref{thm-C-smooth-general} and \ref{thm-C-smooth-a2=0} in Section~\ref{sect-Lipschitz}, Theorems~\ref{thm-Hardy-star}, \ref{thm-close-to-convex-Hardy}, \ref{Hardy-convex-one-direction} and \ref{thm-f-Hardy-typ-real} in Section~\ref{sect-subclases-Hardy}, and Theorem~\ref{thm-Lohwater-Pir-Rudin} in Section~\ref{sect-Hardy-derivative}. Theorem~\ref{thm-a=0-CO-FMR} is given a new proof in an appendix in Section~\ref{sect-CO-FMR-measures}.

\section{Preliminaries} 
\noindent \textbf{Hardy spaces.} We will make frequent use of the fact that 
$$
z\mapsto \frac{1}{(1-z)^a} \;\, \text{is in} \;\, H^p \quad \text{ if and only } \quad  p<\frac{1}{a} 
$$
(see Exercise 1 in \cite[Chap.~1]{Du1}) and of the following two classical theorems, the first one of Hardy and Littlewood (see \cite[\S 5.6]{Du1}) and the second of Prawitz (see \cite[\S 3.7]{Du1} or \cite[\S 2.10]{Du2}). 

\begin{theoremO} \label{thm-Hardy-Littl}
If $f'\in H^p$ for some $p<1$ then $f\in H^q$ for $q=p/(1-p)$.
\end{theoremO}

\begin{theoremO} \label{thm-Prawitz} 
If $f\in S$ and $p\in(0,\infty)$ then 
$$
M_p^p(r,f) \leq p \int_0^r \frac{1}{t} M_{\infty}^p(t,f)  dt, \qquad r\in(0,1). 
$$
\end{theoremO}

A direct corollary of Theorem~\ref{thm-Prawitz}, via the growth theorem, is that $S\subset H^p$ for all $p<1/2$.

\vskip.2cm 
\noindent \textbf{The Carath\'eodory class.} Let $\cP$ denote the Carath\'eodory class of functions $h(z) = 1+  \sum_{n=1}^{\infty} c_n z^n$ which are analytic in $\SD$ and satisfy ${\rm Re}\, h(z) > 0$ for all $z \in \SD$. It is well known that $\cP\subset H^p$ for all $p<1$; see Exercise 2 in \cite[Chap.~1]{Du1}. According to the Herglotz representation, for each function $h$ in $\cP$ there exists a unique probability measure $\mu$ supported on $\ST=\partial \SD$, such that 
$$
h(z) = \int_\ST \frac{1+\la z}{1-\la z} d\mu (\la), \quad z \in \SD.
$$
A theorem of Carath\'eodory states that $|c_n|\leq 2$, for all $n\geq1$. We will be needing the case of equality only for $n=1$: we have that  $|c_1|= 2$ only for $h(z)= \frac{1+\la z}{1-\la z}$, for $\la\in\ST$. See \cite[\S 1.9 and \S 2.5]{Du2}.

\section{Convex functions, Hardy spaces and coefficients} \label{sect-convex-Hardy-coeff}
To prove the Hardy space estimate in $C$ we will combine Prawitz' theorem with the following theorem of Gronwall, which gives growth and distortion bounds in $C$ in terms of $|a_2|$. Although Gronwall \cite{Gr20} provided no proof, it was subsequently proved by Finkelstein \cite{Fi67} (via the Schwarz-Pick Lemma) and by Szynal \emph{et.~al.} \cite{SSW80} (via L\"owner theory).  

\begin{theoremO}[\cite{Gr20, Fi67, SSW80}] \label{thm-growth-dist-C}
If $f\in C$ and $\al = |a_2|$ then 
$$
\ell_{\al}(r) \leq |f(z)| \leq s_{\al}(r), \qquad r=|z|, 
$$
and
$$
\ell'_{\al}(r) \leq |f'(z)| \leq s'_{\al}(r), \qquad r=|z|,
$$
where $s_\al$ is given in \eqref{formula-sector} and
$$ 
\ell_{\al}(z) = \frac{i}{ 2 \sqrt{1-\al^2} }\log\frac{ 1-\la z }{ 1-\overline{\la} z }, \qquad \la=\al+ i \sqrt{1-\al^2}, 
$$
maps onto a shifted infinite vertical strip. 

\end{theoremO}

We now prove Theorem~\ref{thm-convex-deriv-Hardy}, which states that if $f\in C$ then $f'\in H^p$ for all $p<\frac{1}{1+|a_2|}$ and $f\in H^q$ for all $q<\frac{1}{|a_2|}$. 

\begin{proof}[Proof of Theorem~\ref{thm-convex-deriv-Hardy}]
In view of Theorem~\ref{thm-growth-dist-C}, we have that 
$$
|f'(z)| \leq s'_{\al}(r) = \frac{1}{(1-r)^{1+\al} (1+r)^{1-\al} },   \qquad r=|z|, 
$$
where $\al=|a_2|$. By Alexander's theorem (see \cite[\S 2.5]{Du2}), the function $g(z)=zf'(z)$ is strarlike. Applying Prawitz' Theorem~\ref{thm-Prawitz} to $g$ we obtain 
$$
M_p^p(r,g) \leq p \int_0^r \frac{dt}{t^{1-p}(1-t)^{p(1+\al)}}. 
$$
Therefore, we have that 
$$
\lim_{r\to1^-} M_p^p(r,g) \leq p \int_0^1 \frac{dt}{t^{1-p}(1-t)^{p(1+\al)}}, 
$$
which is finite since $1-p<1$ and $p(1+\al)<1$, so that $g\in H^p$. Hence, $f'(z)=g(z)/z$ is also in $H^p$. The estimate for $f$ follows from Theorem~\ref{thm-Hardy-Littl}.

Finally, both estimates are sharp for each $\al\in[0,1]$ in view of the sector function \eqref{formula-sector}. 

\end{proof}

An alternative proof of the estimate for the function $f$ could be given by using the Schwarz-Pick Lemma in order to show that 
$$
(1-|z|^2)\left| \frac{f''(z)}{f'(z)} \right| \leq  2(1+|a_2|), \qquad z\in\SD, 
$$
and then invoking a result of Kim and Sugawa \cite{KS09} to deduce that $f \in H^q$ for $q<\frac{1}{|a_2|}$. We leave the details of this to the interested reader. 

\vskip.2cm
We turn to the coefficients of functions in $C$ and prove Theorem~\ref{thm-conv-n-coeff}.

\begin{proof}[Proof of Theorem~\ref{thm-conv-n-coeff}]
We first prove the asymptotic estimate $a_n\in O(n^{|a_2|-1})$. As in the proof of Theorem~\ref{thm-convex-deriv-Hardy}, we see that the starlike function $g(z)=zf'(z)$ satisfies 
$$
|g(z)| \leq \frac{1}{(1-|z|)^{1+|a_2|}},   \qquad z\in\SD, 
$$
for some $c>0$, in view of Theorem~\ref{thm-growth-dist-C}. Writing $g(z)=\sum_{n=1}^\infty b_nz^n$ we see from Pommerenke's \cite[Theorem 3]{Po62} that the coefficients $b_n = n a_n$ satisfy $b_n \in O(n^{|a_2|})$. From this the estimate follows. 

For the sharpness we see from the sector function~\eqref{formula-sector}, that the estimate cannot be replaced by $o(n^{|a_2|-1})$. Indeed, for $\al=0$ we have that  
$$
s_0(z) = \frac{1}{2} \log \frac{1+z}{1-z} = \sum_{k=0}^\infty \frac{z^{2k+1}}{2k+1}, 
$$
so that $ n a_n \nrightarrow 0$, while for $\al\in(0,1]$ it has been computed in \cite{CHLM23} that the coefficients of $s_\al$ satisfy 
$$
\lim_{n\to\infty} n^{1-\al} a_n = \frac{2^{\al-1}}{\Gamma(\al+1)}>0, 
$$
where $\Gamma$ denotes the gamma function.

We now prove the bound $|a_n| \leq \exp\left( \frac{|a_2|^2-1}{2} \right)$. Recall that according to the Marx-Strohh\"acker theorem if $f\in C$ then $f$ is starlike of order $1/2$, \emph{i.e.}, it satisfies
$$
{\rm Re}\, \frac{zf'(z)}{f(z)} > \frac{1}{2}, \qquad z\in\SD 
$$
(see exercise 23 in \cite[Ch.~2]{Du2}). We set 
$$
h(z)=2 \, \frac{zf'(z)}{f(z)}-1 = 1+ \sum_{n=1}^\infty c_n z^n, \qquad z\in\SD, 
$$ 
which belongs to the class $\cP$. We have that $c_1=2a_2$. We write 
$$
\left( \log\frac{f(z)}{z} \right)' = \frac{h(z)-1}{2z},
$$
which upon integration gives
$$
\sum_{n=0}^\infty a_{n+1} z^n = \frac{f(z)}{z} = \exp \left( \sum_{n=1}^\infty \frac{c_n}{2n} z^n \right). 
$$
An application of the third Lebedev-Milin inequality \cite[\S 5.1]{Du2} yields
\be \label{leb-mil-coef-bound}
|a_{n+1}|^2 \leq \exp \left( \sum_{n=1}^n \frac{|c_k|^2 -4 }{4k} \right), \qquad n \geq 1. 
\ee
Observe that, by Carath\'eodory's theorem, each summand is non-positive and, therefore, we have that 
$$
|a_{n+1}|^2 \leq \exp \left( \frac{|c_1|^2 -4 }{4} \right) = \exp \left( |a_2|^2 -1  \right). 
$$
The proof is complete. 
\end{proof}

\vskip.2cm
Note that keeping more terms in \eqref{leb-mil-coef-bound} would result in a stronger, though less elegant inequality. 

\vskip.2cm
We make a brief remark regarding summability. In view of Theorem~\ref{thm-conv-n-coeff}, the coefficients $a_n$ of functions in $C$ converge to zero, except for the case of the half-plane function and its rotations. It follows from a result of Lewis \cite{Lew83} that, in general, the sequence of coefficients $(a_n)_{n=1}^\infty$ will not belong to any $\ell^p$ space. Indeed, he showed that the polylogarithm function    
$$
z\mapsto \sum_{n=1}^\infty \frac{z^n}{n^t}  
$$  
belongs to $C$ for all $t\geq 0$; observe that  $(n^{-1/p})_{n=1}^\infty \notin \ell^p$ for $p>0$. However, by a theorem of Pommerenke \cite{Po62} we have that $(a_n)_{n=1}^\infty \in \ell^1$ for bounded convex functions. This is best possible since Lewis' example is bounded for $t>1$.

\vskip.2cm
We will be needing the following lemma. 

\begin{lemma} \label{lem-coeff-3-conv} 
If $f\in C$ then $|a_3| \leq \frac{1+2|a_2|^2}{3}$. Equality holds only for rotations of the sector function \eqref{formula-sector}. 

\end{lemma}
\begin{proof}
By a result of Hummel~\cite{Hu57} (see also Trimble~\cite{Tr75} for a simpler proof), we have that
\be \label{ineq-Humm-Trimble}
|a_3 - a_2^2| \leq \frac{1 - |a_2|^2}{3}. 
\ee
A simple application of the triangle inequality
\be \label{ineq-triangle-Trimble}
|a_3| \leq |a_3 - a_2^2| + |a_2|^2 \leq \frac{1 +2 |a_2|^2}{3} 
\ee
gives the desired inequality. It is easy to see that equality holds for rotations of the sector function \eqref{formula-sector}, but some more work is needed in order to show that there are no other extremal functions. 

Let $f$ be an extremal function for the statement. We associate with $f$ a function $h(z)=1+\sum_{n=1}^\infty c_n z^n$ in $\cP$ by 
$$
h(z) = 1+\frac{z f''(z)}{f'(z)}, \qquad z\in\SD, 
$$
(see \cite[\S 2.5]{Du2}) and write the relations for the initial coefficients
$$
c_1 = 2a_2 \qquad \text{and} \qquad c_2 = 6a_3 - 4a_2^2. 
$$
We see that \eqref{ineq-Humm-Trimble} is equivalent to 
$$
\left|c_2 - \frac{c_1^2}{2} \right| \leq 2 -  \frac{|c_1|^2}{2}, 
$$
which is equivalent to the second Carath\'eodory-Toeplitz determinant of $h$ being non-negative. Equality holds here only if the support of the Herglotz measure of $h$ consists of at most 2 atoms (see Theorem IV.22 in \cite[Ch.~IV, \S 7]{Tsu}). Hence, we have that
$$
h(z) = t \frac{1+\la z}{1- \la z} + (1-t) \frac{1+\mu z}{1- \mu z}, \qquad t\in[0,1], \;\la, \mu \in \ST
$$
(these correspond to extremals that have not been considered by Trimble). Without loss of generality we may assume that $t \in[\frac{1}{2},1]$. The coefficients of $h$ are $c_n = 2t\la^n + 2(1-t)\mu^n$, while a simple calculation gives us that 
$$
6( a_3 - a_2^2 )= c_2 - \frac{c_1^2}{2}  =   2 t(1-t)  (\la  -\mu )^2
$$
and 
$$
a_2^2 = \frac{c_1^2}{4} =  \big[ \mu+ t(\la  -\mu )\big]^2. 
$$
The case when $\la=\mu$ corresponds to rotations of the half-plane function $s_1(z)= \frac{z}{1 - z}$. Otherwise, if $\la\neq\mu$ we have that equality in the the triangle inequality of \eqref{ineq-triangle-Trimble} shows that these two expressions must have the same argument and therefore 
$$
0< \frac{\big[ \mu+ t(\la  -\mu )\big]^2}{(\la  -\mu )^2} = \left( t + \frac{1}{\la \overline{\mu} -1 } \right)^2, 
$$
from which it follows that $\la \overline{\mu} \in \SR$, so that $\la \overline{\mu}=-1$. Hence, we have that 
$$
\big(\log f'(z) \big)' = \frac{h(z)-1}{z} =  \frac{2 t \la}{1-\la z} -  \frac{2(1-t)\la }{1+\la z}, 
$$
which yields 
$$
f'(z) = \frac{1}{ (1-\la z)^{(1+\al)} (1+\la z)^{(1-\al)} }, 
$$
readily showing that $f(z) = \overline{\la} s_\al(\la z)$, with $\al=2t-1$. 
\end{proof}

\vskip.2cm
We turn to the lower order $\beta$ of functions in $C$ and first prove the following. 

\begin{lemma} \label{lem-low-ord-conv}
If $f\in C$ then $\beta=|a_2|$ if and only if either $a_2=0$ or $f$ is a rotation of $s_\al$.
\end{lemma}
\begin{proof}
It is evident that if $a_2=0$ then $\beta=0$ as well. We see that the sector function $s_\al$ satisfies $\frac{s_\al''(z)}{s_\al'(z)} = \frac{2(\al+z)}{1-z^2}$, so that after a straightforward calculation we get that 
$$
|A_{s_\al}(z)|^2 = \al^2 +  \frac{4(1-\al^2)y^2}{|1-z^2|^2}, \qquad z=x+yi. 
$$
Hence $|A_{s_\al}(z)|\geq \al$ and therefore $\beta=\al$ for $s_\al$. Clearly the same is true for rotations of $s_\al$. 

On the other hand, for an arbitrary $f\in C$ and for $z=re^{i\t}$ we find that
$$
|A_f(z)|^2 = |a_2|^2 +  2 r {\rm Re}\big[ ( 3a_3 \overline{a_2} -2a_2^2\overline{a_2} - a_2)  e^{i\t} \big]+ O(r^2), \qquad r\to 0. 
$$
If $\beta=|a_2|$ then $|A_f(z)| \geq |a_2|$ for every $\t$, so that $3a_3 \overline{a_2} -2a_2^2\overline{a_2} - a_2=0$. It follows that either $a_2=0$ or equality holds in Lemma~\ref{lem-coeff-3-conv}, that is, $f$ is a rotation of a sector function. 
\end{proof}

\vskip.2cm
Now we prove Theorem~\ref{thm-convex-Hardy-low-ord}, according to which for $f\in C$ we have that $f'\in H^p$ for all $p<\frac{1}{1+\beta}$ and $f\in H^q$ for all $q<\frac{1}{\beta}$. Its proof relies on Theorem~\ref{thm-convex-deriv-Hardy}. 

\begin{proof}[Proof of Theorem~\ref{thm-convex-Hardy-low-ord}]
Let $p_0<\frac{1}{1+\beta}$ and write $\ve = \frac{1}{p_0} - (1+\beta)>0$. Let $\ze\in\SD$ be such that $|A_f(\ze)|<\beta+\ve$ and write $F_{\ze}$ for the corresponding Koebe transform \eqref{form-Koebe-transf}. By Theorem~\ref{thm-convex-deriv-Hardy}, $F'_{\ze}$ belongs to $H^p$ for every $p< (1+|A_f(\ze)|)^{-1}$. Since
$$
|f'\circ \psi(z)| \leq 4 |f'(\ze)| \, |F'_{\ze}(z)|, \qquad \text{where} \qquad \psi(z) = \frac{\ze + z }{1+\overline{\ze} z}, 
$$
we have that $f'\circ \psi$ belongs to the same Hardy spaces as $F'_{\ze}$. An appeal to the Corollary in \cite[\S 2.6]{Du1} shows that $f'\in H^p$ for every $p< (1+|A_f(\ze)|)^{-1}$. Since $p_0=\frac{1}{1+\beta+\ve}$ lies in this range, the first assertion in the statement has been proved. The second assertion follows immediately from Theorem~\ref{thm-Hardy-Littl}. 

\end{proof}

\section{Convex functions and the angle at infinity} \label{sect-convex-angle}
We begin by giving an equivalent definition for the angle at infinity by making use of that fact that half-tangents exist at all boundary points of a convex domain; see \cite[\S 3.5]{Po92}. Let $\Om$ be an unbounded convex domain, $f$ be one of its Riemann mappings and assume that $f(e^{it_0})=\infty$ for some $t_0\in\SR$. We set 
\be \label{form-half-tang}
\vt^{\pm} = \lim_{ t \to t_0^{\pm} } \arg f(e^{it})
\ee
to be the argument of the two half-tangents at infinity. We will prove that the difference $ \Delta\vt = \vt^+ - \vt^-$ is equal to the angle at infinity.

\begin{lemma} \label{lem-half-tang}
$\Theta = \Delta\vt$
\end{lemma}
\begin{proof}
We may translate in order to have $0\in \Om $ and rotate so that 
$$
\vt := \vt^+ = -\vt^- \in [0,\tfrac{\pi}{2}], 
$$
in which case $\Delta\vt=2\vt$. 

\vskip.1cm
\noindent \underline{Step (i): $\Theta \leq \Delta\vt$}. Let $\ve>0$ and consider the sectors 
$$
P_x = \left\{z\in\SC : |\arg(z-x)|< \vt+ \tfrac{\ve}{2} \right\}, \qquad x<0. 
$$
We will show that there exists $x_0<0$ sufficiently small so that $\Om \subset P_{x_0}$. If this is not true then let $w$ be the point leftmost in $\partial \Om \cap \partial P_x$, with ${\rm Im}\, w >0$, and see that the half-line $\ell$ that extends from $w$ to infinity along $\partial P_x$ cannot lie entirely within $\Om$ since 
$$
\lim_{ \ell \ni z \to \infty } \arg z = \vt+ \tfrac{\ve}{2}
$$ 
is larger than $\vt$. Hence, by convexity, we have that $\Om \cap \ell $ must be a single open line segment. Now it is evident that sliding the sector to the left we can obtain $P_{x_0} \supset \Om $ for some $x_0<0$ sufficiently small. Hence, $\Theta \leq \Delta\vt + \ve$ and letting $\ve\to0$ we get that $\Theta \leq \Delta\vt$. 

\vskip.1cm
\noindent \underline{Step (ii): $\Theta \geq \Delta\vt$}. We begin with the case $\Theta =0$ and assume that $\vt>0$ in order to get a contradiction. Then $\Om$ lies in some symmetric with respect to the real line sector $P$ of aperture $\Delta\vt/3$. But this is impossible since for points $z$ in the upper half plane 
$$
\lim_{ \partial P \ni z \to \infty } \arg z   = \tfrac{\vt}{3} < \vt = \lim_{ \partial \Om \ni z \to \infty } \arg z. 
$$
Hence $\vt=0$ in this case. 

We assume that $\Theta \in(0,\pi]$. If $\vt=0$ then we work as in step (i) in order to find some sector of aperture $\Theta/2$ that contains $\Om$, a contradiction. Hence, $\vt>0$. Let $\ve>0$ and consider  
$$
Q = \left\{z\in\SC : |\arg z|< \vt- \tfrac{\ve}{2} \right\}. 
$$
Now if the ray $\partial Q \cap \{{\rm Im}\, z >0\}$ intersects $\partial \Om$ then, by convexity, this can only occur at a unique point, so that the ray lies outside $\Om$ at a neighborhood of infinity. But this is impossible since $\arg z = \vt- \tfrac{\ve}{2}$ on this ray, a contradiction. Hence $Q\subset\Om$, so that $\Delta\vt - \ve \leq \Theta$, which gives us the desired conclusion upon letting $\ve\to0$. 
\end{proof}

Let $f\in C$ and see that by the Herglotz formula there exists a unique probability measure $\mu$, supported on $\ST$, such that 
\be \label{Herglotz-C}
1+\frac{zf''(z)}{f'(z)} = \int_{\ST} \frac{1+\la z}{1-\la z} d\mu(\la).
\ee
If $\la_0\in\ST$ is mapped by $f$ to infinity then $\mu$ has a point mass there with $\mu(\la_0	)\geq \frac{1}{2}$, a fact that is usually attributed to Paatero \cite{Pa31}. By convexity, there can be either none, one or two preimages of infinity, the latter only in the case of an infinite strip. 

It is well known that the boundary rotation, \textit{i.e.}, the net change of the direction angle of the tangent, corresponding to an arc $I$ on $\ST$ under a function $f\in C$ equals $2\pi \mu(I)$. To see this,  write $I=(a,b)$ and compute the limit of 
$$
\arg \frac{\partial}{\partial\t} f(re^{i\t}) \bigg|_a^b = \t + \arg f'(re^{i\t}) \bigg|_a^b 
$$
as $r\to1^-$ by making use of either formula (6) in \cite{EK70} or formula (14) in \cite[page 62]{Po92}.

\begin{lemma}\label{lem-angle-point-mass}
Let $f\in C$ with associated measure $\mu$ and assume that $f(\la_0) =\infty$ for some $\la_0 \in\ST$. Then 
$$
\Theta\big(f(\SD)\big) =  \big(2\mu(\la_0)-1\big) \pi. 
$$
\end{lemma}

\begin{proof}
In case the range of $f$ is an infinite strip the desired equality is evident since $\Theta=0$ and $\mu(\la_0)=\frac{1}{2}$. We assume that $\la_0$ is the unique preimage of infinity. The boundary rotation of $f$ on the finite plane is equal to 
$$
2\pi\mu\big(\ST\backslash\{\la_0\}\big) = 2\pi\big(1-\mu( \la_0 )\big).  
$$
On the other hand, in view of the half-tangents \eqref{form-half-tang}, the boundary rotation of $f$ on the finite plane is equal to 
$$
\vt^- - (\vt^+-\pi) = \pi - \Delta \vt = \pi-\Theta,
$$
by Lemma~\ref{lem-half-tang}. Equating the two the desired equality follows. 

\end{proof}

We now prove Theorem~\ref{thm-angle-l-order}, according to which $\Theta\big(f(\SD)\big) = \beta \pi$ for $f\in C$.

\begin{proof}[Proof of Theorem~\ref{thm-angle-l-order}] 
If $f$ is bounded then $\beta=0$ by a result of Pommerenke \cite{Po08} and, moreover, we clearly have that $\Theta=0$, so that the desired conclusion holds in this case. We assume that $f$ is unbounded and, after applying a rotation, we may further assume that $f(1)=\infty$. Using \eqref{Herglotz-C} we compute
\begin{align*}
A_f(z) & =  \frac{1}{2} (1-|z|^2) \frac{f''(z)}{f'(z)} - \overline{z} \\ 
& =   \int_{\ST} \frac{(1-|z|^2)\la}{1-\la z} \, d\mu(\la) - \int_{\ST} \overline{z} \, d\mu(\la)\\ 
& =   \int_{\ST} \frac{\la - \overline{z} }{1-\la z} \, d\mu(\la).  
\end{align*}
Applying the triangle inequality twice we obtain 
\begin{align*}
| A_f(z)| & = \left| \frac{1-\overline{z}}{1-z} \, \mu(1) + \int_{\ST\backslash\{1\}} \frac{\la( 1-  \overline{\la z})}{1-\la z} \, d\mu(\la) \right| \\
& \geq \mu(1) - \left| \int_{\ST\backslash\{1\}} \frac{\la( 1-  \overline{\la z})}{1-\la z} \, d\mu(\la) \right| \\
& \geq \mu(1) - \mu(\ST\backslash\{1\}) \\ 
& = 2 \mu(1) -1 \\
& = \frac{\Theta}{\pi},  
\end{align*}
by Lemma~\ref{lem-angle-point-mass}. Now, for $x\in(0,1)$ we have that 
$$
A_f(x)  = \int_{\ST} \frac{\la - x }{1-\la x} \, d\mu(\la)  = \mu(1) + \int_{\ST\backslash\{1\}} \frac{\la -x }{1-\la x} \, d\mu(\la), 
$$
where the integrand in the last expression is uniformly bounded and for each $\la\in\ST\backslash\{1\}$ it converges to $-1$ when $x\to1^-$. By the dominated convergence theorem we have that 
$$
\lim_{x\to1^-} A_f(x) = \mu(1) - \mu(\ST\backslash\{1\}) = \frac{\Theta}{\pi}. 
$$ 
The proof is complete. 

\end{proof}

We mention here, without providing the details, that a slightly more involved proof of Theorem~\ref{thm-angle-l-order} would be more revealing: the Koebe transform of $f$, as its parameter tends to the preimage of infinity, converges locally uniformly in $\SD$ to a sector mapping of aperture $\Theta$, whose second coefficient is equal to $\frac{\Theta}{\pi}$. We chose the proof we presented in favor of simplicity.

\section{Convex functions and smoothness at the boundary} \label{sect-Lipschitz}
We will be needing a theorem of Hardy and Littlewood, according to which for a function $f\in H^p$ we have that 
\be \label{eq-thm-HL-smooth}
f(e^{i\t}) \in \Lambda_t^p \qquad \Longleftrightarrow \qquad M_p(r,f') \in O\big( (1-r)^{t-1} \big), \quad r\to1;
\ee
see Theorem~5.4 in \cite[\S 5.2]{Du1}. 

\begin{proof}[Proof of Theorem~\ref{thm-C-smooth-general}]
As in the proof of Theorem~\ref{thm-convex-deriv-Hardy}, we use the distortion estimate from Theorem~\ref{thm-growth-dist-C} and apply Prawitz' Theorem~\ref{thm-Prawitz} to the starlike function $g(z)=zf'(z)$ in order to deduce 
$$
M_p^p(r,g) \leq p \int_0^r \frac{1}{t} M_{\infty}^p(t,g)  dt  \leq p \int_0^r \frac{t^{p-1} \, dt}{(1-t)^{p(1+|a_2|)}} \leq \frac{ c }{ (1-r)^{p(1+|a_2|)-1}}, 
$$
for some constant $c>0$. Hence
$$
M_p(r,f') \in O\left( \frac{ 1 }{ (1-r)^{1+|a_2|-\frac{1}{p}}}  \right), \qquad r\to1. 
$$
An application of \eqref{eq-thm-HL-smooth} yields the desired result since for $t= \frac{1}{p}-|a_2|$, $p<1/|a_2|$ implies $t>0$, while $p\geq 1$ implies $t \leq 1$. 

For the sharpness we fix $\al \in [0,1)$ and $1\leq p < \frac{1}{\al}$, and see that the boundary function of $s_\al$, the sector function \eqref{formula-sector}, cannot belong to $\Lambda_t^p$, with $t= \frac{1}{p}-\al+\ve$, for any $\ve>0$. If it did, then by \eqref{eq-thm-HL-smooth} we would have that
$$
M_p(r,s_\al') \leq \frac{ c_1 }{ (1-r)^{1-\frac{1}{p}+\al-\ve}}, 
$$
for some $c_1>0$. Since
$$
s_\al'(z) = \frac{1}{ (1-z)^{1+\al} (1+z)^{1-\al} }, 
$$
we see that the above would imply 
\begin{align*}
\frac{ c_1^p }{ (1-r)^{(1+\al-\ve)p-1}} & \geq \int_0^{2\pi} \frac{d\t}{ |1-z|^{(1+\al)p} |1+z|^{(1-\al)p} }, \qquad z=re^{i\t} \\ 
& \geq c_2 \int_0^{\pi} \frac{d\t}{ |1-z|^{(1+\al)p} } \\
&    =  c_2 \int_0^{\pi} \frac{d\t}{ [ (1-r)^2+4r\sin^2(\t/2) ]^{ \frac{(1+\al)p}{2}} } \\
& \geq c_2 \int_0^{\pi} \frac{d\t}{ [ (1-r)^2+r\t^2 ]^{ \frac{(1+\al)p}{2}} } \\
& = \frac{ c_2  }{ (1-r)^{(1+\al)p-1}} \int_0^{ \frac{\sqrt{r}\pi }{ 1-r } } \frac{dt}{  \sqrt{r} [ 1+t^2 ]^{ \frac{(1+\al)p}{2}} }  \\ 
& > \frac{ c_2  }{ (1-r)^{(1+\al)p-1}} \int_0^{ \frac{ 1 }{ 1-r } } \frac{dt}{ [ 1+t^2 ]^{ \frac{(1+\al)p}{2}} }, 
\end{align*}
after the change of variabless $\sqrt{r} \t = (1-r) t$, for $r$ sufficiently close to 1 and for some $c_2>0$. But this leads to 
$$
0 < \int_0^{ \frac{ 1 }{ 1-r } } \frac{dt}{ [ 1+t^2 ]^{ \frac{(1+\al)p}{2}} }  \leq  c_3 (1-r)^{\ve p} \to 0, \qquad \text{as} \;\; r\to1, 
$$
for some $c_3>0$, a contradiction. 
\end{proof}

We now assume that $a_2=0$ and first prove the following lemma. Recall that, in this case, $|a_3|\leq1/3$ by Lemma~\ref{lem-coeff-3-conv}. 

\begin{lemma} \label{lem-C-distort-zero}
Let $f\in C$ with $a_2=0$. Then
$$
|f'(z)| \leq (1-|z|)^{ - \frac{2(1+3|a_3|)}{3(1+|a_3|)} }, \qquad z\in\SD.
$$
\end{lemma}
\begin{proof}
Let $\ga=3|a_3|$. According to Lemma~\ref{lem-coeff-3-conv}, we have that $\ga=1$ only for rotations of the function $s_0(z)= \frac{1}{ 2 }\log\frac{ 1+ z }{ 1- z }$, which satisfy the statement. 

We assume that $\ga<1$. The well-known characterization of the class $C$ tells us that
$$
1+\frac{z f''(z)}{f'(z)} = \frac{1+z\vp(z)}{1-z\vp(z)}, \qquad z\in\SD,
$$
for some analytic $\vp:\SD\to \overline{\SD}$. Writing 
$$
\frac{f''(z)}{f'(z)} = \frac{2\vp(z)}{1-z\vp(z)}
$$
we see that $\vp(0) = a_2 = 0$, so that $\vp(z) = z \psi(z)$ for some $\psi:\SD\to \overline{\SD}$. A standard calculation shows that $\psi(0)=3 a_3$. Upon integrating and taking real parts we obtain
$$
\log |f'(z)| = 2\,  {\rm Re} \int_{[0,z]} \frac{\ze\psi(\ze)}{1-\ze^2\psi(\ze)} d\ze \leq  2 \int_{[0,z]} \left| \frac{\ze\psi(\ze)}{1-\ze^2\psi(\ze)} \right| |d\ze|, \qquad z=re^{i\t}.  
$$
A simple consequence of the Schwarz-Pick lemma is that 
\be \label{ineq-Schw-Pick-consq}
|\psi(\ze)| \leq \frac{\ga + \rho }{1+\ga \rho }, \qquad \rho=|\ze|<1, 
\ee
which leads to 
\begin{align*}
\left| \frac{\ze\psi(\ze)}{1-\ze^2\psi(\ze)} \right| \leq &  \frac{ \rho | \psi(\ze) | }{ 1-\rho^2|\psi(\ze)|} \\ 
 \leq & \frac{ \rho(\ga +\rho) }{(1-\rho) [ 1+(1+\ga)\rho +\rho^2 ] } \\
= &  \frac{1}{3+\ga} \left( \frac{1+\ga}{1-\rho}  - \frac{1+\ga +2\rho}{1+(1+\ga)\rho +\rho^2}  \right) \\ 
< & \frac{1+\ga}{(3+\ga)(1-\rho)}
\end{align*}
Therefore, we get that 
$$ 
\log |f'(z)| < \frac{2(1+\ga)}{3+\ga} \int_0^r \frac{d\rho}{1-\rho} = \frac{2(1+\ga)}{3+\ga}  \log\frac{1}{1-r}, 
$$
from which the desired inequality follows. 

\end{proof}

\begin{proof}[Proof of Theorem~\ref{thm-C-smooth-a2=0}]

If $f$ is not a rotation of $s_0(z) = \frac{1}{2}\log\frac{1+z}{1-z}$ then $\ga=3|a_3|<1$, hence $t=\frac{1-\ga}{3+\ga} \in (0,\frac{1}{3}]$, and the statement follows by a direct application of Lemma~\ref{lem-C-distort-zero} and \cite[Theorem~5.1 in \S 5.2]{Du1}. 

\end{proof}

\section{The Hardy space for other geometric subclasses} \label{sect-subclases-Hardy}
Recall that $S^*$ stands for the class of functions in $S$ whose range is starlike with respect to the origin. We prove Theorem~\ref{thm-Hardy-star}: if $f\in S^*$ then $f'\in H^p$ for all $p<\frac{2}{4+|a_2|}$ and $f\in H^q$ for all $q<\frac{2}{2+|a_2|}$.  

\begin{proof}[Proof of Theorem~\ref{thm-Hardy-star}]
By Alexander's theorem we may write $f(z)=zh'(z)$ for some $h\in C$. We compute $f'(z) = h'(z) \big[1+z\frac{h''(z)}{h'(z)} \big]$. Also, we write $h(z) = \sum_{n=1}^{\infty} c_n z^n$ and find that $a_2=2c_2$. Setting $\al=|a_2|$ and applying H\"older's inequality with the conjugate parameters 
$$
s=\frac{4+\al}{2+\al} \qquad \text{and} \qquad t=2+\frac{\al}{2}
$$
we get 
$$
\int_0^{2\pi} |f'(e^{i\t})|^p  \tfrac{d\t}{2\pi}  \leq \left( \int_0^{2\pi} |h'(e^{i\t})|^{ps} \tfrac{d\t}{2\pi} \right)^{ \frac{1}{s} } \left( \int_0^{2\pi} \left|1+e^{i\t} \frac{h''(e^{i\t})}{h'(e^{i\t})}\right|^{pt} \tfrac{d\t}{2\pi} \right)^{ \frac{1}{t} }.  
$$
Now the first term is finite in view of Theorem~\ref{thm-convex-deriv-Hardy} and the fact that $ps<\frac{2}{2+\al}=\frac{1}{1+|c_2|}$, while the second term is finite since $pt<1$ and the integrand has positive real part. The estimate for $f$ follows from Theorem~\ref{thm-Hardy-Littl}. 

For the sharpness, we see that if $h=s_\ga$ is a sector function \eqref{formula-sector} with $\ga\in[0,1]$ then the corresponding starlike function via Alexander's theorem is 
$$
f(z) = \frac{z}{ (1-z)^{1+\frac{\al}{2}} (1+z)^{1-\frac{\al}{2}} }, 
$$
with $a_2=\al=2\ga$. This functions does not belong to $H^{ \frac{2}{2+\al} }$ and its derivative is 
$$
f'(z) = \frac{1+\al z + z^2}{ (1-z)^{2+\frac{\al}{2}} (1+z)^{2-\frac{\al}{2}} }, 
$$
which does not belong to $H^{ \frac{2}{4+\al} }$. 
\end{proof}

Geometrically, the starlike function proving the sharpness in Theorem~\ref{thm-Hardy-star} maps $\SD$ onto the complement of two slits symmetric with respect to the real axis lying on the rays $\arg w = \pm \frac{\pi}{4}(2+\al)$. The modulus of the endpoints of these two rays is $ (2-\al)^{ - \frac{2-\al}{4} } (2+\al)^{ - \frac{2+\al}{4} }$.

\vskip.2cm
We turn to the class $K$ of functions $f$ in $S$ which are close-to-convex, \emph{i.e.}, that satisfy 
\be \label{repr-close-to-conv}
{\rm Re} \left(  \la  \frac{f'(z)}{g'(z)} \right) > 0, \qquad z\in\SD, 
\ee
for some $g\in C$ and $\la\in\ST$. It is known that the class $K$ coincides with the class of \emph{linearly accessible} functions, that is, functions whose range is the complement of the union of mutually disjoint half-lines; see \cite[\S 2.6]{Du2}. 

Using Baernstein's star function, Leung \cite{Leu79} proved that every $f\in K$ satisfies 
$$
M_p(r,f') \leq M_p(r,k'), \qquad r<1, \, p>0, 
$$
where $k$ is the Koebe function (see also \cite[\S 7.5]{Du2}). Hence, for every $f\in K$, the derivative $f'$ belongs to $H^p$ for $p<1/3$, and this is sharp since $k'(z)=\frac{1+z}{(1-z)^3}$ does not belong to $H^{1/3}$. The inclusion $K\subset H^p$ for $p<1/2$, derived from Prawitz' theorem, is sharp because $k\in K$. We now prove Theorem~\ref{thm-close-to-convex-Hardy}, according to which the above Hardy space estimates are sharp whenever $a_2\neq 0$, while they can be improved to $f'\in H^{1/3}$ and $f\in H^{1/2}$ in the case when $a_2=0$. 

\begin{proof} [Proof of Theorem~\ref{thm-close-to-convex-Hardy}]
In view of \eqref{repr-close-to-conv} we may write $H = \la f' / g'$ and $H(0)=\la = a+ib$, and see that $a\in (0,1]$. We normalize so that $h = (H-ib)/a$ belongs to the Carath\'eodory class $\cP$. Setting
$$
f(z) = \sum_{n=1}^{\infty} a_n z^n, \qquad  g(z) = \sum_{n=1}^{\infty} b_n z^n \qquad \text{and} \qquad  h(z) = \sum_{n=0}^{\infty} c_n z^n, 
$$
we get the relation $a_2 = b_2 + a \overline{\la}c_1/2$ between the first non-trivial coefficients. 

The estimates for the case $a_2\neq0$ have already been discussed and so it remains to prove their sharpness. We let $\la=1$, $g$ to be the half-plane function $s_1$ and $h$ to be a function whose Herglotz measure has point masses $t$ and $1-t$, for $t\in(0,1]$, at the points 1 and $-1$, respectively. Hence, we have that 
$$
g(z)=\frac{z}{1- z} \qquad \text{and} \qquad  h(z) = t \, \frac{1+ z}{1- z} + (1-t) \, \frac{1 - z}{1 + z}, 
$$
which shows that $b_2=1$ and $c_1=4t-2$, so that $a_2 = 2t \in(0,2]$. A simple calculation gives 
$$
f'(z) = t \, \frac{1+ z}{(1- z)^3} + (1-t) \, \frac{1}{1 - z^2}, 
$$
which does not belong to $H^{1/3}$. An integration shows that 
\be \label{extr-funct-K}
f(z) = t \, \frac{z}{(1- z)^2} + \frac{1-t}{2} \, \log\frac{1+z}{1 - z}, 
\ee
which does not belong to $H^{1/2}$. 

\vskip.1cm
We turn to claim (i), the case $a_2=0$. It holds that $c_1=-2\la b_2/a$. We distinguish two cases. First, if $|b_2|=1$, say $b_2=\mu \in\ST$, then $ |c_1| = 2/a \geq 2 $ and since, also, $|c_1| \leq 2$ by Carathéodory's theorem, we get that $a=1$, $\la =1$ and, therefore, $g$ and $h$ must be the half-plane functions
$$
g(z)=\frac{z}{1-\mu z} \qquad \text{and} \qquad  h(z)=\frac{1-\mu z}{1+\mu z}. 
$$
It follows that $f'(z)=\frac{1}{1-\mu^2 z^2}$, which belongs to $H^p$ for all $p<1$. For the case $|b_2| <1$, we set $\beta=|b_2|$ and use H\"older's inequality
$$
\int_{\ST} |f'|^p \leq \left( \int_{\ST} |g'|^{pr} \right)^{1/r} \left( \int_{\ST} |H|^{ps} \right)^{1/s} 
$$
with the conjugate exponents 
$$
r = \frac{2+\beta}{1+\beta} \qquad \text{and} \qquad  s = 2+\beta. 
$$
The right-hand side is bounded whenever $p<\frac{1}{2+\beta}$, since $ps <1$ and $pr<\frac{1}{1+\beta}$, so that Theorem~\ref{thm-convex-deriv-Hardy} can be applied. Since the exponent $p=1/3$ lies in this range, it follows that $f'\in H^{1/3}$. Now the claim that $f\in H^{1/2}$ follows from Theorem~\ref{thm-Hardy-Littl}.

To show the sharpness of (i) we let $\la=1$, $g=s_\beta$ with $\beta\in[0,1)$, and $h\in P$ with a measure consisting of three point-masses: $t, t$ and $1-2t$ at the points $\mu= e^{i\t}, \overline{\mu}$ and 1, respectively, given by
$$
h(z) = t \, \frac{1+\mu z}{1-\mu z} + t \, \frac{1+\overline{\mu} z}{1- \overline{\mu} z} + (1-2t)\frac{1+z}{1-z}, \qquad  t\in\left(0,\tfrac{1}{2}\right), \, \t\in(0,\pi). 
$$
We compute $c_1 = 4t \cos\t +2(1-2t)$ and 
\be \label{form-f'-close-to-conv}
f'(z) = \frac{ 1 }{ (1-z)^{1+\beta}(1+z)^{1-\beta} } \left[t \, \frac{1+\mu z}{1-\mu z} + t \, \frac{1+\overline{\mu} z}{1- \overline{\mu} z} + (1-2t)\frac{1+z}{1-z} \right].  
\ee
Assume in order to get a contradiction that $f$ belongs to $H^q$ for some $q\in(\frac{1}{2},1)$ and recall the growth estimate (see \cite[\S 3.2]{Du1}): 
$$
|f(z)| \leq \frac{c}{(1-r)^{1/q}}, \qquad r=|z|, 
$$
for some $c>0$. Now the bound
$$
|f'(z)| \leq \frac{\hat{c}}{(1-r)^{1+1/q}}, \qquad r=|z|, 
$$
for some $\hat{c}>0$, follows from the Cauchy integral formula. Note that $b_2 = \beta$ and let $\beta$ be such that $\frac{1}{2} < \frac{1}{1+\beta} < q$. We ensure that $a_2=0$ by choosing, say, $t=\frac{3+\beta}{8}$ and $\cos\t = - \frac{1+3\beta}{3+\beta}$, so that 
$$
c_1 = -2 \beta = -2 b_2
$$
is satisfied. We see from \eqref{form-f'-close-to-conv} that the order of growth of $f'$ at $z=1$ is $\beta + 2 > 1+1/q$, which together with the above estimate for $f'$ leads to a contradiction.

Finally, if $f'\in H^p$ for some $p>1/3$ then by Theorem~\ref{thm-Hardy-Littl} we would have that $f\in H^q$ for $q=\frac{p}{1-p} > 1/2$, a contradiction, as we just saw. 

\end{proof}

\vskip.2cm
Recall from the introduction the subclasses $R$ and $R^+ $ of $S$ consisting of functions whose range is convex in one direction and convex in the positive direction, respectively. Since $R^+ \subset  R \subset K$ the estimates in Theorem~\ref{thm-close-to-convex-Hardy} apply to functions in $R^+$ and $R$ and, as already mentioned, when $a_2\neq 0$  these are optimal (they are $f'\in H^p$ for $p<1/3$ and $f\in H^q$ for $q<1/2$). This can be seen from the function \eqref{extr-funct-K}, used in the proof of Theorem~\ref{thm-close-to-convex-Hardy}, which has $a_2 = 2t \in(0,2]$ and maps $\SD$ onto the complement of two horizontal slits at heights $\pm (1-t)\frac{\pi}{4}$ whose real part extends from $-\infty$ to some point $x_0=x_0(t)$ and, thus, belongs to $R^+$. 

We now prove Theorem~\ref{Hardy-convex-one-direction}, according to which if a function $f\in S$ satisfies $a_2=0$ then 
\begin{itemize}
\item[(i)] if $f\in R$ then $f'\in H^p$ for all $p<1/2$ and $f\in H^q$ for all $q<1$, while 
\item[(ii)] if $f\in R^+$ then $f'\in H^p$ for all $p<1$ and $f\in H^q$ for all $q<\infty$.   
\end{itemize}

\begin{proof}[Proof of Theorem~\ref{Hardy-convex-one-direction}]
For the case (i), let $f\in R$ have $a_2=0$ and apply a rotation, if necessary, in order to get that $f(\SD)$ is convex in the vertical direction. Then, by a theorem of Royster and Ziegler \cite{RZ76} there exist parameters $\mu\in\ST$ and $n\in[0,\pi]$ such that
$$ 
{\rm Re}\big[ -i\mu (1-2\cos(n) \overline{\mu} z + \overline{\mu}^2 z^2) f'(z) \big] \geq 0,  \qquad z\in\SD. 
$$ 
Consider the function $g(z) = \overline{\mu} f(\mu z)$. We set 
\be \label{form-def-h-g'}
h(z)=-i\mu (1-2\cos(n) z + z^2) g'(z)
\ee 
and note that it satisfies ${\rm Re} \, h \geq 0$ in $\SD$. Set $\mu=a+ib$. Since $h(0)=-i\mu =b-ia$ we get that $b\in[0,1]$. If $b=0$ then by the open mapping theorem $h\equiv -ia$ and we find that 
$$
g'(z) = \frac{1}{1-2\cos(n) z + z^2}. 
$$
Since $g''(0)=0$ we see that $\cos(n)=0$, hence $g'(z) = \frac{1}{1+ z^2}$, which belongs to $H^p$ for $p<1$. 

Assume now that $b>0$. It follows from \eqref{form-def-h-g'} that $h'(0) = 2i\mu\cos(n)$. Since $(h+ia)/b$ is in $\cP$, by Carathéodory's theorem we have that $|h'(0)|\leq 2b$, and therefore that $|\cos(n)|\leq b$. We distinguish two cases. First, if $|\cos(n)|=1$ then also $b=1$, hence $a=0$ and $\mu=i$. By the case of equality in Carathéodory's theorem we get that 
$$
h(z) = \frac{1+\la z}{ 1- \la z} =1 +2\la z + \ldots
$$
for some $\la\in\ST$. Hence, $\la=-\cos(n)=\pm 1$. In view of \eqref{form-def-h-g'} we have that 
$$
h(z) = (1+\la z)^2 g'(z)
$$
and, therefore, 
$$
g'(z) = \frac{1}{ 1- \la^2 z^2} =\frac{1}{ 1-  z^2}. 
$$
Now clearly $g'\in H^p$ for all $p<1$. For the remaining case of $|\cos(n)|<1$ we write 
$$
1-2\cos(n) z + z^2 = (\nu-z)(\overline{\nu}-z), \qquad \text{with} \qquad \nu=e^{in} \in \ST\backslash\{\pm1\}, 
$$ 
which shows that this expression cannot be a perfect square. Therefore, by the Cauchy-Scwharz inequality we have that
$$
\int_{\ST} |g'|^p \leq  \left( \int_{\ST} |h|^{2p} \right)^{1/2} \left( \int_{\ST} \frac{1}{|\nu-z|^{2p}|\overline{\nu}-z|^{2p}} \right)^{1/2} < \infty
$$
whenever $p<1/2$. Therefore, we have that $g'\in H^p$ for $p<1/2$ in all subcases of case (i), hence the same is true for $f'$. By Theorem~\ref{thm-Hardy-Littl} we get the desired estimate for $f$. 

The sharpness of case (i) can readily be seen from the function 
$$
f(z) =\frac{z}{1+z^2} = z - z^3 + \ldots, 
$$
whose image domain is $\SC \backslash \{x : |x| \geq 1/2\}$, and which does not belong to $H^{1}$, nor $f'(z) =\frac{1-z^2}{(1+z^2)^2}$ belongs to $H^{1/2}$. Also, clearly, $f\in R\backslash R^+$.

\vskip.1cm
Turning to case (ii), we let $f\in R^+$ with $a_2=0$. For each $z$ in $\SD$ the limit 
$$
\tau = \lim_{t\to\infty} f^{-1}\big( f(z) +t \big)  
$$
exists, belongs to $\ST$ and is independent of $z$; see \cite[p.440]{EKRS10}, for example.  It is well known that the function 
$$
g(z) =  (\tau -z)(1-\overline{\tau}z) f'(z) 
$$
satisfies ${\rm Re}\,g(z) >0$; see \cite[\S 3.5]{ES10}. Since $g(0) = \tau = a + ib$, with $a\in(0,1]$, we may normalize to get that 
$$
h(z)=\frac{g(z)-ib}{a} = 1 + \sum_{n=1}^\infty c_n z^n 
$$ 
belongs to $\cP$. Relating the first coefficients we find that $ac_1=2 \tau a_2-2=-2$, since $a_2=0$. Now, $2=a|c_1| \leq |c_1| \leq 2$, which implies that $a=1$ and $c_1=-2$. Hence, we have that $h(z)=\frac{1-z}{1+z}$ by the case of equality in Carathéodory's theorem. It is therefore clear that 
$$
f'(z) =  \frac{1}{1- z^2}  \qquad \text{and} \qquad f(z) = \frac{1}{2} \, \log\frac{1+z}{1 - z}, 
$$
which proves our claim and its sharpness. 

\end{proof}

\vskip.2cm
Let $ S_\SR$ denote the class of functions in $S$ that have real coefficients and let $T$ denote the class of typically real functions, \emph{i.e.}, functions $f$ which are holomorphic in $\SD$, are normalized by $f(0)=f'(0)-1=0$, have real values on the interval $(-1,1)$ and non-real values elsewhere in the disk. Clearly $S_\SR \subset T$. We now prove Theorem~\ref{thm-f-Hardy-typ-real}, which states that $T \subset H^p$ for all $p<1/2$, and that this cannot be improved in $S_\SR$, for any $|a_2| \in [0,2]$. 

\begin{proof}[Proof of Theorem~\ref{thm-f-Hardy-typ-real}]
Let $f\in T$ and use Rogosinski's representation $f(z) = \frac{z}{1-z^2} \, h(z)$, where $h\in\cP$ and has real coefficients; see \cite[\S 2.8]{Du1}. An application of the Cauchy-Scwharz inequality shows that
$$
\int_{\ST} |f|^p \leq \left( \int_{\ST} \frac{1}{|1-z^2|^{2p} } \right)^{1/2}  \left( \int_{\ST} |h|^{2p} \right)^{1/2} < \infty
$$
whenever $p<1/2$.

For the second claim and the case $\al \in (0,2]$ we can consider the function \eqref{extr-funct-K}. However, we provide a proof for the whole interval $[0,2]$ by using the extremal functions in Jenkins' \cite{Je54} solution of the Gronwall problem and an observation of Hayman~\cite[p.262]{Hay55}: for each $\al \in [0,2]$ there exists $f\in S_\SR$ with $a_2=\al$ and such that 
$$
\lim_{r\to1} (1-r)^2 M_{\infty}(r,f) = \ga, 
$$
where 
$$
\ga = 4 \la^2 e^{2-4\la} \qquad \text{and} \qquad  \la = \frac{1}{2-\sqrt{2-\al}};  
$$
see also \cite[\S 5.5]{Du2}. If it were true that $f\in H^{1/2}$ then 
$$
\int_0^1 M^{1/2}_{\infty}(r,f) dr < \infty
$$
by an inequality of Hardy and Littlewood (see \cite[Theorem~A]{BGP04} for the optimal version of this inequality). However, this is readily contradicted by the fact that $\ga>0$, which finishes the proof.

\end{proof}

\section{Integrability of the derivative of univalent functions} \label{sect-Hardy-derivative}

We now prove Theorem~\ref{thm-Lohwater-Pir-Rudin}, according to which for every $\al \in[0,2)$ there exists $f \in S_\SR$ such that $a_2 =\al$ and its derivative $f'$ does not belong to the Nevanlinna class $N$. 

\begin{proof}[Proof of Theorem~\ref{thm-Lohwater-Pir-Rudin}]
We keep the notation given in the introduction for the function $\Phi \in H^\infty \cap S_\SR$ constructed in \cite{LPR55}, whose derivative $\Phi'$ has radial limits almost nowhere, so that $\Phi'\notin N$. Recall that the second coefficient of $\Phi$ is equal to $1/4$. 

For $\al =0$ we consider the square root transformation $f(z) = \sqrt{ \Phi(z^2) }$ which produces an odd function in $S_\SR$. If $f'$ had radial limits in some open subset of $\ST$ then so would $\Phi'$ since $f(z) f'(z) = z \Phi'(z^2)$, a contradiction. Hence $f'\notin N$.

For $\al \in(0,2)$ we consider the composition $f = k_r \circ g$, where $k_r(z) = \frac{1}{r} k(rz)$, with $r<1$, is a dilation of the Koebe function $k(z)=\frac{z}{(1-z)^2}$, and $g$ is the integral transform 
$$
g(z) = \int_0^z \Phi'(\ze)^\ve d\ze, \qquad z\in\SD, 
$$
which is univalent for $0<\ve \leq 1/4$ by a theorem of Pfaltzgraff \cite{Pf75}. To show that the composition is well defined we prove that for each $r<1$ there exists $\ve>0$ sufficiently small so that the range of $g$ lies in the disk $|z|<1/r$. Indeed, by the distortion theorem (see \cite[\S 2.3]{Du2}) we have that 
\begin{align*}
|g(z)| & \leq |z| \int_0^1 |\Phi'(tz)|^\ve \, dt \\ 
 & \leq |z| \int_0^1 \frac{ (1+t|z|)^\ve }{ (1-t|z|)^{3\ve} } \, dt \\ 
&  < 2^\ve |z| \int_0^1 \frac{ dt }{ (1-t|z|)^{3\ve} } \\ 
& =  \frac{ 2^\ve }{ 1-3\ve } \, \big[1-(1-|z|)^{1-3\ve} \big] \\ 
& < \frac{ 2^\ve }{ 1-3\ve }. 
\end{align*}
Since the last expression is increasing with $\ve$ there exists a unique $\ve_0=\ve_0(r)>0$ for which $\frac{ 2^{\ve_0} }{ 1-3\ve_0 }=\frac{1}{r}$. It follows that if the parameters lie in 
$$
\Om = \big\{(r,\ve) \, : \, r\in(0,1), \;  0< \ve \leq \min\{\ve_0(r) , \tfrac{1}{4} \} \big\}
$$
we have that $f$ is a well defined function in $S$. Moreover, $f\in S_\SR$ since all functions involved have real coefficients. In view of 
$$
f'(z) = k'\big( rg(z) \big) \Phi'(z)^\ve
$$
it is clear that $f'$ has radial limits almost nowhere, hence $f'\notin N$. 

It remains to show that the equation 
$$
a_2(f) = 2r + \frac{\ve}{4} = \al
$$
has a solution in $\Om$ for each $\al\in(0,1)$. It is not hard to give explicit solutions, but for our purpose it suffices to see that $\ve_0$ is strictly decreasing to $\ve_0(1)=0$ since 
$$
\ve_0'(r) = -\frac{2^{\ve_0}}{ 3+r2^{\ve_0}\log2 } <0,
$$
so that $\Om$ is a connected set, and note that 
$$
\inf_\Om a_2(f) =0 \qquad \text{and} \qquad \sup_\Om a_2(f) =2, 
$$
after which the existence of a solution follows from the intermediate value theorem.

\end{proof}

\section{Appendix: Convex functions and probability measures } \label{sect-CO-FMR-measures}
In view of Theorem~\ref{thm-a=0-CO-FMR} convex functions with vanishing second coefficient are very specific: 
\vskip.15cm
\begin{center}
\textit{If $f\in C$ has $a_2=0$ then $f$ is either bounded or a rotation of $s_0(z)= \frac{1}{2} \log\frac{1+z}{1-z}$. }
\end{center} 
\vskip.15cm
The proof of Theorem~\ref{thm-a=0-CO-FMR} relies on the Schwarzian derivative, the fact that $C$ is included in the Nehari class and on Lemma 4 from \cite{CO95}. Alternatively, it follows from Theorem 3 in \cite{FMR98}, which provides a structural formula for unbounded convex functions. In Theorem~\ref{prop-FMR-measures} we provide an alternative statement and proof of Theorem~\ref{thm-a=0-CO-FMR} using the language of probability measures. 

Let $f\in C$ and $\mu$ be its associated measure via the Herglotz formula \eqref{Herglotz-C}. Note that $a_2=\int_{\ST}\la d\mu(\la)$ and recall that a point $\la_0\in\ST$ is mapped by $f$ to infinity if and only if $\mu$ has a point mass there with $\mu(\la_0)\geq \frac{1}{2}$. Also, let us denote by $\delta_{\la}$ the Dirac measure on $\la$.

\begin{theorem} \label{prop-FMR-measures}
If a probability measure $\mu$ on $\ST$ has vanishing first moment and has a point mass at $\la_0$ with $\mu(\la_0)\geq 1/2$ then 
$$
\mu = \frac{1}{2} (\delta_{\la_0} + \delta_{-\la_0}). 
$$
\end{theorem}

\begin{proof}
Applying a rotation we may assume that $\la_0=1$. Since $\mu$ is a probability measure we have that
$$
\int_{\ST\backslash\{1\}} d\mu(\la) = 1 - \mu(1). 
$$
Its first moment is zero, hence
\begin{align*}
0 & =  {\rm Re} \int_{\ST}\la d\mu(\la) \\ 
 & = \mu(1) +  \int_{(-\pi,\pi]\backslash\{0\}} \cos\t \,  d\mu(\t) \\ 
 & = \mu(1) + \int_{(-\pi,\pi]\backslash\{0\}}   \big[ 2\cos^2(\t/2) -1 \big] d\mu(\t) \\ 
 & =  2\mu(1) -1 + \int_{(-\pi,\pi]\backslash\{0\}}  2 \cos^2(\t/2) \, d\mu(\t) \\ 
 & \geq  2\mu(1) -1. 
\end{align*}
It follows that $\mu(1)=1/2$. 

Once again, the vanishing of the first moment yields 
$$
0 =  {\rm Re} \int_{\ST}\la d\mu(\la) =  2 \int_{(-\pi,\pi]}  \cos^2(\t/2) \, d\mu(\t)  -1. 
$$
Hence, 
$$
\frac{1}{2} =  \int_{(-\pi,\pi]}  \cos^2(\t/2) \, d\mu(\t) = \mu(1) + \int_{(-\pi,\pi]\backslash\{0\}} \cos^2(\t/2) \, d\mu(\t), 
$$
so that 
$$
\int_{(-\pi,\pi]\backslash\{0\}} \cos^2(\t/2) \, d\mu(\t) = 0. 
$$
Since the integrand is non-negative, the measure $\mu$ on $(-\pi,\pi]\backslash\{0\}$ can only be supported where $\cos(\t/2)$ vanishes, that is, at $\t=\pi$. Hence, ${\rm supp} (\mu) = \{\pm 1\}$ and the proof is complete. 

\end{proof}

\vskip.3cm
\noindent \textbf{Conflict of interest statement}. On behalf of all authors, the corresponding author states that there is no conflict of interest.

\vskip.2cm
\noindent  \textbf{Data availability statement}. This manuscript has no associated data.

\end{document}